\documentclass[10pt,a4paper]{article}
\usepackage[margin=1.1cm]{geometry}
\usepackage{amsmath,amsthm,amsfonts,amssymb,amscd,cite,graphicx}

\usepackage{tikz}
\usepackage{url}

\usepackage{titlesec}
\titleformat{\section}
{\normalfont\fontsize{12}{15}\bfseries}{\thesection}{1em.}{}

\usepackage[labelfont=bf]{caption}
\counterwithin{figure}{section}
\counterwithin{table}{section}

\newtheorem{corollary}{Corollary}[section]
\newtheorem{lemma}{Lemma}[section]
\newtheorem{definition}{Definition}[section]
\newtheorem{theorem}{Theorem}[section]
\newtheorem{remark}{Remark}[section]

\allowdisplaybreaks[4]

\let\oldbibliography\thebibliography
\renewcommand{\thebibliography}[1]{%
	\oldbibliography{#1}%
	\setlength{\itemsep}{-2pt}%
}

\title{Self-avoiding polygons on a three-row square-lattice strip}

\author{Michael von Thaden\\[1mm]
	\small Fachhochschule Westküste\\
	\small Fritz-Thiedemann-Ring 20, 25746 Heide, Germany\\
	\small \texttt{vonthaden@fh-westkueste.de}}

\date{}

\begin{document}

	\maketitle
\begin{abstract}
We give a closed formula for the number $p^{S_2}(n)$ of self-avoiding polygons (SAPs) of length $n$ on the strip $S_2:=\mathbb{Z}\times\{0,1,2\}$, together with closed formulas for those subtypes of SAPs which are determined by the numbers of vertical steps in their leftmost and rightmost columns. For the subtype whose leftmost and rightmost columns each contain two vertical steps, we also derive an alternative representation as a binomial sum. Our derivation is elementary: it is purely combinatorial and geometric and avoids generating
functions. Comparing the two representations yields a new geometric proof of an
identity arising in Larsen's treatment \cite{L07} of a problem posed by Gessel \cite{G95}. Finally, we show that this subtype of SAPs is closely connected to the sequence A007909. More precisely, for $m\geq0$, the number of these SAPs whose leftmost and
rightmost columns each contain two vertical steps and whose length equals $2m+6$ is given by the term of this sequence with index $m$, which thereby acquires a geometric interpretation alongside the compositions it enumerates.

\end{abstract}

\section{Introduction}

A \emph{self-avoiding walk} (SAW) on the square lattice is a nearest-neighbour path that does not visit any vertex more than once. A \emph{self-avoiding polygon} (SAP) is a closed nearest-neighbour walk that does not visit any vertex more than once, except that its last
vertex coincides with its origin.  Self-avoiding walks and polygons sit at a crossroads of combinatorics, statistical physics and polymer science. They play a central role in statistical physics, in particular as models for long-chain polymers (see \cite{MS13} and \cite{GJ09}). At the same time, their definition is elementary. They are therefore easily accessible and also of pedagogical interest, since many of the basic questions can be understood without heavy technical prerequisites.

One of the naturally arising and central questions is: How many SAPs exist for a given perimeter $n$?  Interestingly, although self-avoiding walks  and polygons are very easy to define, closed enumerative formulas exist only in limiting or highly restricted cases. Up to now, closed formulas for their otherwise unrestricted enumeration only seem available for narrow strips of the square lattice. For the strip $S_1:=\mathbb{Z}\times\{0,1\}$, Zeilberger and Benjamin obtained a beautiful connection between the sequence $c^{S_1}_n (n>1)$ of self-avoiding walks of length $n$  and the Fibonacci sequence: $c^{S_1}_n=8F_n-\delta_n$, where $\delta_n=4$ for odd $n$ and $\delta_n=n$ for even $n$ (see \cite{Z96} and \cite{B06}). On the other hand, the situation for SAPs on the very same strip is as easy as it gets. We have $p^{S_1}(n) = 1$ for every even integer $n\geq 4$.

In this paper we enumerate self-avoiding polygons on the strip $S_2:=\mathbb{Z}\times\{0,1,2\}$, up to horizontal translation. We will do so by analyzing self-avoiding polygons of different left and right boundary configurations, where we look at the number of vertical steps we find at the very leftmost and rightmost positions of the SAP. Whereas Bousquet-M\'elou and Brak \cite{BMB09} have already treated self-avoiding polygons in this setting and provided the generating functions for the number of SAPs, the proofs for these formulas in this paper will be elementary, combinatorial and geometrical, no generating functions will be needed. Furthermore, the resulting formulas will directly relate the number of self-avoiding polygons to the sequence A007909. Comparing the two representations gives a geometric proof of the unweighted binomial identity evaluated by Larsen \cite[p.~38, Remark~4.3]{L07} in the course of solving Gessel's related weighted problem \cite{G95}.

\section{The Self-Avoiding Polygon on the Strip $S_2$}
Since most of the self-avoiding polygons considered in this paper are confined between two horizontal boundary lines, and since the main case of interest is the three-row strip of geometric width $2$, we introduce the following notation.

\begin{definition}
	For $L\in\mathbb N_0$, let
	$$
	S_L := \mathbb Z\times\{0,1,\ldots,L\}.
	$$
	We call $S_L$ the horizontal strip of width $L$ in the square lattice.
	In particular, $S_2=\mathbb Z\times\{0,1,2\}$.
\end{definition}

Throughout this paper, we regard a SAP as an unoriented simple cycle and two SAPs are counted as the same if their edge sets differ by a horizontal translation. Reflected polygons are counted separately.

\begin{remark}\label{slices}
We will use a few easy observations concerning self-avoiding polygons on $S_2$, which do not deserve separate lemmas but, because we use them so frequently, shall be stated and explained here. First, let us observe that every unit-width slice of $S_2$ between two consecutive columns within the horizontal span of an SAP contains exactly two of the polygon’s horizontal edges. This is simply because the alternatives of just one or three horizontal steps are incompatible with the SAP being closed. In particular, every SAP on $S_2$ resembles the bit of a key; see Figure~\ref{fig:key-bit-sap-46}.

\begin{figure}[t]
	\centering
	\begin{tikzpicture}[x=0.7cm,y=0.7cm,line cap=round]
		\tikzset{
			gridline/.style={line width=0.8pt, draw=black!70},
			walk/.style={line width=2.4pt, draw=red!85},
		}
		
		\draw[gridline] (-0.5,0) -- (17.5,0);
		\draw[gridline] (-0.5,1) -- (17.5,1);
		\draw[gridline] (-0.5,2) -- (17.5,2);
		
		\draw[walk]
		(0,0) -- (1,0) -- (2,0) -- (3,0) -- (3,1) -- (4,1)
		-- (5,1) -- (5,0) -- (6,0) -- (7,0) -- (8,0) -- (9,0)
		-- (10,0) -- (11,0) -- (11,1) -- (12,1) -- (13,1)
		-- (13,0) -- (14,0) -- (15,0) -- (16,0) -- (17,0)
		-- (17,1) -- (17,2) -- (16,2) -- (16,1) -- (15,1)
		-- (14,1) -- (14,2) -- (13,2) -- (12,2) -- (11,2)
		-- (10,2) -- (9,2) -- (9,1) -- (8,1) -- (7,1) -- (7,2)
		-- (6,2) -- (5,2) -- (4,2) -- (3,2) -- (2,2) -- (1,2)
		-- (0,2) -- (0,1) -- cycle;
		
	\end{tikzpicture}
	\caption{An SAP resembling a key bit.}
	\label{fig:key-bit-sap-46}
\end{figure}
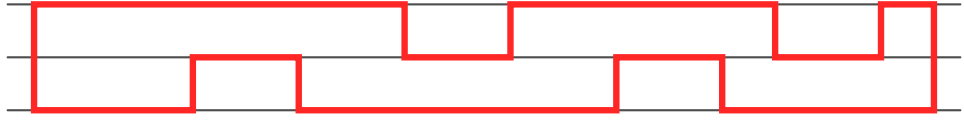
 Furthermore, we observe that the length $n$ of every SAP is always an even number, because an SAP always returns to its origin, which means there must be the same number of vertical steps going ``up'' as steps going ``down''. The same holds true for the number of horizontal steps going to the ``right'' and to the ``left''.  
\end{remark}

For our approach to work, we  will have to split the number of all self-avoiding polygons into ones with different left and right boundary configurations. What we mean by this, will be clarified by the following definition

\begin{definition}
We shall say that a self-avoiding polygon starts/ends with a \emph{type-1 configuration}, if the leftmost/rightmost points on the polygon constitute a single vertical step. On the other hand do we say that a self-avoiding polygon starts/ends with a \emph{type-2 configuration}, if the leftmost/rightmost points on the polygon constitute two successive vertical steps. We shall call SAPs, which end with a type-1 configuration respectively a type-2 configuration, \emph{type-1 SAPs respectively type-2 SAPs}. 
	
Furthermore, $p^{S_2}_{xy}(n)$ with $x,y \in \{1,2\}$ shall be defined as the number of self-avoiding polygons on  the Strip  $S_2$ which start with a type-$x$ and end with a type-$y$ configuration. 
\end{definition}

In Figure \ref{fig:sap-type-configurations} two SAPs are shown, the one on the left starts and ends with a type-1 configuration, whereas the one on the right starts and ends with a type-2 configuration.

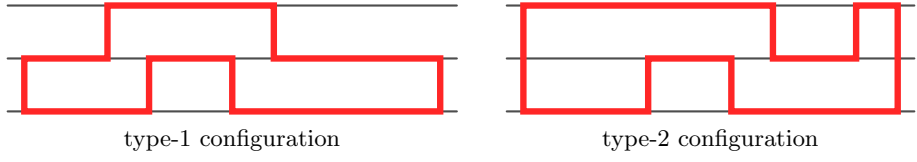
\begin{figure}[t]
	\centering
	\begin{tikzpicture}[x=0.55cm,y=0.7cm,line cap=round]
		\tikzset{
			gridline/.style={line width=0.8pt, draw=black!70},
			walk/.style={line width=2.4pt, draw=red!85},
		}
		
		\begin{scope}
			\draw[gridline] (-0.4,0) -- (10.4,0);
			\draw[gridline] (-0.4,1) -- (10.4,1);
			\draw[gridline] (-0.4,2) -- (10.4,2);
			
			\draw[walk]
			(0,0) -- (1,0) -- (2,0) -- (3,0) -- (3,1)
			-- (4,1) -- (5,1) -- (5,0) -- (6,0) -- (7,0)
			-- (8,0) -- (9,0) -- (10,0) -- (10,1) -- (9,1)
			-- (8,1) -- (7,1) -- (6,1) -- (6,2) -- (5,2)
			-- (4,2) -- (3,2) -- (2,2) -- (2,1) -- (1,1)
			-- (0,1) -- cycle;
			
			\node[font=\small] at (5,-0.55)
			{type-$1$ configuration};
		\end{scope}
		
		\begin{scope}[shift={(12,0)}]
			\draw[gridline] (-0.4,0) -- (9.4,0);
			\draw[gridline] (-0.4,1) -- (9.4,1);
			\draw[gridline] (-0.4,2) -- (9.4,2);
			
			\draw[walk]
			(0,0) -- (1,0) -- (2,0) -- (3,0) -- (3,1)
			-- (4,1) -- (5,1) -- (5,0) -- (6,0) -- (7,0)
			-- (8,0) -- (9,0) -- (9,1) -- (9,2) -- (8,2)
			-- (8,1) -- (7,1) -- (6,1) -- (6,2) -- (5,2)
			-- (4,2) -- (3,2) -- (2,2) -- (1,2) -- (0,2)
			-- (0,1) -- cycle;
			
			\node[font=\small] at (4.5,-0.55)
			{type-$2$ configuration};
		\end{scope}
		
	\end{tikzpicture}
	\caption{SAPs with type-$1$ and type-$2$ configurations,
		respectively, at both extreme $x$-coordinates.}
	\label{fig:sap-type-configurations}
\end{figure}

Now, we are ready to formulate the following lemma, which will help us to derive recursive formulas for  the different, mutually exclusive subtypes of SAPs on the strip $S_2$, which have  different boundary configurations. 

\begin{lemma}\label{lem:peeling}
	For every $x\in\{1,2\}$ and every even integer $n\geq 8$,
	\begin{align}
		p^{S_2}_{x1}(n)
		&=p^{S_2}_{x1}(n-2)+2p^{S_2}_{x2}(n-2),
		\label{eq:peel-1}\\
		p^{S_2}_{x2}(n)
		&=p^{S_2}_{x1}(n-4)+p^{S_2}_{x2}(n-2).
		\label{eq:peel-2}
	\end{align}
\end{lemma}

\begin{proof}
First, let us remark that a SAP of length $n>6$ necessarily has horizontal span at least $2$, which means that the SAP consists of more than one unit-width slice---a fact that we will implicitly use in our proof. All operations considered below affect only the right end and therefore leave the type-$x$ configuration at the left end unchanged.
	
Let us start with a self-avoiding polygon which ends with a type-1 configuration and has length $n\geq 8$.
	
Now, if we ``peel off'' the rightmost unit-width slice of this polygon,	we must distinguish two cases. Either the preceding unit-width slice to the left contains two horizontal steps lying on the two extreme rows $0$ and $2$ (Case 1), or they lie on the same two rows as in the rightmost slice (Case 2).
	
In the first case, the resulting open polygonal path can simply be closed and transformed into a standard SAP of length $n-2$ by deleting three steps and inserting one vertical step; see Figure~\ref{fig:right-configuration-deletions}(a) for an illustration of this procedure. It is important to point out that exactly two type-1 SAPs of length $n$ are mapped onto each distinct type-2 SAP of length $n-2$: the first one has its two rightmost horizontal steps on rows $0$ and $1$, and the second one has the corresponding steps on rows $1$ and $2$.
	
In the second case, almost the same is true as in the previous case: the resulting open polygonal path can simply be closed and transformed into a standard SAP of length $n-2$ by deleting three steps and inserting one vertical step. In contrast to the previous case, however, this operation establishes a bijection between the type-1 SAPs of length $n$ belonging to this case and the type-1 SAPs of length $n-2$, since every resulting SAP admits a unique inverse extension. Overall, this results in the recursive relation
$$
p^{S_2}_{x1}(n) = p^{S_2}_{x1}(n-2)+2p^{S_2}_{x2}(n-2).
$$
	
The proof for SAPs which end with a type-2 configuration works in the same way as the one before. If the preceding unit-width slice contains two horizontal steps lying on the two extreme rows $0$ and $2$, peeling off the rightmost slice maps the type-2 SAP of length $n$ onto a type-2 of length $n-2$. Otherwise, it maps the type-2 SAP of length $n$ onto a type-1 SAP of length $n-4$; see Figure~\ref{fig:right-configuration-deletions}(b) for an illustration of this procedure. In both cases, the operation is bijective, since every resulting SAP admits a unique inverse extension. Thus, in this case, we have
$$
p^{S_2}_{x2}(n) = p^{S_2}_{x1}(n-4)+p^{S_2}_{x2}(n-2),
$$
which proves the lemma.
\end{proof}

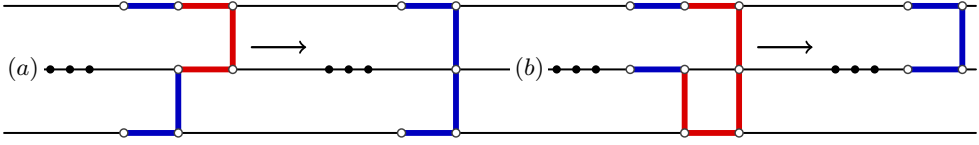
\begin{figure}[t]
	\centering
	\begin{tikzpicture}[
		x=0.72cm,
		y=0.84cm,
		line cap=round,
		line join=round
		]
		\def\showmarks{1}
		
		\tikzset{
			gridline/.style={
				line width=0.75pt,
				draw=black
			},
			blueedge/.style={
				line width=2.2pt,
				draw=blue!75!black
			},
			rededge/.style={
				line width=2.2pt,
				draw=red!85!black
			},
			stepmark/.style={
				circle,
				draw=black!75,
				fill=white,
				line width=0.55pt,
				inner sep=0pt,
				minimum size=3pt
			},
		}
		
		\draw[gridline] (-1.45,0) -- (16.40,0);
		\draw[gridline] (-1.45,1) -- (16.40,1);
		\draw[gridline] (-1.45,2) -- (16.40,2);


		\begin{scope}
			
			\node[
			font=\small,
			fill=white,
			inner sep=1.5pt
			] at (-1.10,1) {$(a)$};
			
			\foreach \x in {-0.60,-0.24,0.12}
			\fill[black] (\x,1) circle (1.55pt);
			
			\draw[blueedge]
			(0.75,0) -- (1.75,0) -- (1.75,1);
			
			\draw[blueedge]
			(0.75,2) -- (1.75,2);
			
			\draw[rededge]
			(1.75,1) -- (2.75,1)
			-- (2.75,2) -- (1.75,2);
			
			\ifnum\showmarks=1
			\foreach \x/\y in {
				0.75/0,1.75/0,1.75/1,2.75/1,
				2.75/2,1.75/2,0.75/2
			}
			\node[stepmark] at (\x,\y) {};
			\fi
			
			\draw[->,line width=0.85pt]
			(3.10,1.35) -- (4.10,1.35);
			
			\foreach \x in {4.52,4.88,5.24}
			\fill[black] (\x,1) circle (1.55pt);
			
			\draw[blueedge]
			(5.85,0) -- (6.85,0)
			-- (6.85,1) -- (6.85,2)
			-- (5.85,2);
			
			\ifnum\showmarks=1
			\foreach \x/\y in {
				5.85/0,6.85/0,6.85/1,6.85/2,5.85/2
			}
			\node[stepmark] at (\x,\y) {};
			\fi
			
		\end{scope}


		\begin{scope}[shift={(9.30,0)}]
			
			\node[
			font=\small,
			fill=white,
			inner sep=1.5pt
			] at (-1.10,1) {$(b)$};
			
			\foreach \x in {-0.60,-0.24,0.12}
			\fill[black] (\x,1) circle (1.55pt);
			
			\draw[blueedge]
			(0.75,1) -- (1.75,1);
			
			\draw[blueedge]
			(0.75,2) -- (1.75,2);
			
			\draw[rededge]
			(1.75,1) -- (1.75,0)
			-- (2.75,0) -- (2.75,1)
			-- (2.75,2) -- (1.75,2);
			
			\ifnum\showmarks=1
			\foreach \x/\y in {
				0.75/1,1.75/1,1.75/0,2.75/0,
				2.75/1,2.75/2,1.75/2,0.75/2
			}
			\node[stepmark] at (\x,\y) {};
			\fi
			
			\draw[->,line width=0.85pt]
			(3.10,1.35) -- (4.10,1.35);
			
			\foreach \x in {4.52,4.88,5.24}
			\fill[black] (\x,1) circle (1.55pt);
			
			\draw[blueedge]
			(5.85,1) -- (6.85,1)
			-- (6.85,2) -- (5.85,2);
			
			\ifnum\showmarks=1
			\foreach \x/\y in {
				5.85/1,6.85/1,6.85/2,5.85/2
			}
			\node[stepmark] at (\x,\y) {};
			\fi
			
		\end{scope}
		
	\end{tikzpicture}
	
\caption{Deletion of the two possible right-hand configurations
		and closure of the resulting open polygonal paths.}
	\label{fig:right-configuration-deletions}
\end{figure}

It is important to point out that these recursions are equivalent to the transfer structure implicit in the generating function of Bousquet-Mélou and Brak \cite{BMB09}, from which the closed formula below could also be extracted by routine computations; our derivation is instead purely geometric and combinatorial.

Based upon the above formula, we will derive a closed recurrence for $p^{S_2}_{x1}(n)$ and $p^{S_2}_{x2}(n)$, which directly results in   the closed formulas for these two sequences.

\begin{theorem}\label{Rec}
For every $x\in\{1,2\}$ and every even $n\geq 10$ we have 
$$
p^{S_2}_{x1}(n) = 2p^{S_2}_{x1}(n-2) -p^{S_2}_{x1}(n-4)+2p^{S_2}_{x1}(n-6)
$$
and 
$$
p^{S_2}_{x2}(n) = 2p^{S_2}_{x2}(n-2) - p^{S_2}_{x2}(n-4)+2p^{S_2}_{x2}(n-6)
$$
\end{theorem}

\begin{proof}
Equation \eqref{eq:peel-1} from  Lemma \ref{lem:peeling} can be restated in the form $p^{S_2}_{x2}(n-2) = (p^{S_2}_{x1}(n)-p^{S_2}_{x1}(n-2))/2$. Substituting the last formula  and shifted versions into Equation \eqref{eq:peel-2} from Lemma \ref{lem:peeling} leads to $(p^{S_2}_{x1}(n+2)-p^{S_2}_{x1}(n))/2 = (p^{S_2}_{x1}(n)-p^{S_2}_{x1}(n-2))/2+p^{S_2}_{x1}(n-4)$. Solving for $p^{S_2}_{x1}(n+2)$ directly results in  
$$
p^{S_2}_{x1}(n+2) = 2p^{S_2}_{x1}(n) -p^{S_2}_{x1}(n-2)+2p^{S_2}_{x1}(n-4)
$$
and shifting the index results in  the form stated in the theorem

Furthermore, Equation \eqref{eq:peel-2} can be restated as $p^{S_2}_{x1}(n-4) = p^{S_2}_{x2}(n)-p^{S_2}_{x2}(n-2)$. Substituting this and its shifted version into Equation \eqref{eq:peel-1} and rearranging the terms yields the following formula
$$
p^{S_2}_{x2}(n+4) = 2p^{S_2}_{x2}(n+2) - p^{S_2}_{x2}(n)+2p^{S_2}_{x2}(n-2).
$$ 
Replacing $n$ with $n-4$ yields the second recurrence initially for even $n\geq 12$, since the recurrences in Lemma~\ref{lem:peeling} hold for
$n\geq 8$. It remains to verify the case $n=10$. For $x=1$, the relevant values are
$$
p^{S_2}_{12}(4)=0,\qquad p^{S_2}_{12}(6)=0,\qquad
p^{S_2}_{12}(8)=2,\qquad p^{S_2}_{12}(10)=4,
$$
and for $x=2$ they are
$$
p^{S_2}_{22}(4)=0,\qquad p^{S_2}_{22}(6)=1,\qquad
p^{S_2}_{22}(8)=1,\qquad p^{S_2}_{22}(10)=1.
$$
Thus the second recurrence also holds for $n=10$.
\end{proof}

The above theorem immediately results in  a formula for the number $p^{S_2}(n)$ of all self-avoiding polygons on the strip $S_2$ as well as for all specific sub-types with given start/end configurations. 

\begin{corollary}\label{PT}
For every even integer $n\geq 4$, we have
\begin{align}
		p^{S_2}_{11}(n)
		&= \frac{2^{n/2}+2\rho(n/2-2)}{5},\\
		p^{S_2}_{12}(n)=p^{S_2}_{21}(n)
		&= \frac{2^{n/2-1}+\rho(n/2)+\rho(n/2-1)}{5},\\
		p^{S_2}_{22}(n)
		&= \frac{2^{n/2-2}+\rho(n/2-3)}{5},\\
		p^{S_2}(n)
		&= \frac{9\cdot2^{n/2-2}+\rho(n/2-1)}{5},
\end{align}
where $\rho:\mathbb Z\to\mathbb Z$ denotes the $4$-periodic function whose values for $n\equiv 0,1,2,3\pmod 4$ are $3,1,-3,-1$, respectively.
\end{corollary}

\begin{proof}
For $n=4,6,8$, the initial values of the sequences $p^{S_2}_{11}(n)$, $p^{S_2}_{12}(n)=p^{S_2}_{21}(n)$, and $p^{S_2}_{22}(n)$ are
$$
	\begin{gathered}
		p^{S_2}_{11}(4)=2,\quad
		p^{S_2}_{12}(4)=p^{S_2}_{21}(4)=p^{S_2}_{22}(4)=0,\\
		p^{S_2}_{11}(6)=2,\quad
		p^{S_2}_{12}(6)=p^{S_2}_{21}(6)=0,\quad
		p^{S_2}_{22}(6)=1,\\
		p^{S_2}_{11}(8)=p^{S_2}_{12}(8)=p^{S_2}_{21}(8)=2,\quad
		p^{S_2}_{22}(8)=1.
	\end{gathered}
$$
	
The first three formulas stated in the corollary attain precisely these initial values. Moreover, since
$$
\rho(m+2)=-\rho(m)
$$
for every $m\in\mathbb Z$, direct substitution shows that all three formulas satisfy the recursive relations of Theorem~\ref{Rec}. Since these recurrences, together with the values at $n=4,6,8$, uniquely determine the respective sequences for every even $n\geq4$, the first three formulas follow.
	
Finally,
$$
p^{S_2}(n) = p^{S_2}_{11}(n)+p^{S_2}_{12}(n) +p^{S_2}_{21}(n)+p^{S_2}_{22}(n).
$$
Using the first three formulas and the identity
$$
2\rho(m-2)+2\rho(m)+\rho(m-1)+\rho(m-3)=0,
$$
we obtain
$$
p^{S_2}(n) = \frac{9\cdot2^{n/2-2}+\rho(n/2-1)}{5}.
$$
This proves the corollary.
\end{proof}

\begin{remark}
	Because we have such simple formulas for all the sequences $p^{S_2}_{xy}(n)$ with $x, y \in\{1,2\}$ as well as for $p^{S_2}(n)$ , it is of limited additional interest  to state the connective constant $\mu_{\mathrm{SAP},S_2}$ for self-avoiding polygons on the strip $S_2$. For the sake of completeness, we do nevertheless remark that it directly follows from Corollary \ref{PT} that $\mu_{\mathrm{SAP},S_2} = \lim_{\substack{n\to\infty\\ n\ \text{even}} } (p^{S_2}(n))^{1/n} = \sqrt{2}$.
\end{remark}

Interestingly, one of the formulas above yields, after reindexing, the OEIS sequence A007909 \cite{OEIS}. More precisely, if
$a(m)$ denotes the term of this sequence indexed by $m$, then
$$
p^{S_2}_{22}(2m+6)
=
a(m)
=
\frac{2^{m+1}+\rho(m)}{5}
\qquad (m\geq 0),
$$
where $\rho$ is the $4$-periodic function defined above. The OEIS entry also gives a binomial-sum representation of the same sequence. After adapting its indexing to the current offset in the OEIS, this yields the identity
$$
\sum_{i=0}^{\lfloor m/3\rfloor}
2^i\binom{m-i}{2i}
=
\frac{2^{m+1}+\rho(m)}{5}
\qquad (m\geq 0).
$$
The same unweighted binomial sum is evaluated in \cite[p.~38, Remark~4.3]{L07} in the course of solving Gessel's related weighted problem \cite{G95}. We now give an independent geometric proof of the displayed identity by deriving its left-hand side via a second enumeration of $p^{S_2}_{22}(n)$, the number of SAPs that start and end with a type-2 configuration. Analogous statements also hold true for all SAPs with arbitrary left and right boundary configurations as well as for the overall number of SAPs on the strip $S_2$. The proofs work in the very same way as the one we will provide in the following.

\begin{theorem}\label{Binom}
	For every even integer $n\geq 6$, we have
	$$
	p^{S_2}_{22}(n)
	=
	\sum_{i=0}^{\lfloor (n-6)/6\rfloor}
	2^i\binom{n/2-i-3}{2i}.
	$$
\end{theorem}

\begin{proof}
As we have already mentioned in Remark~\ref{slices}, every SAP on the strip $S_2$ resembles a key bit. Because the SAP under consideration starts and ends with a type-2 configuration, it follows that it contains a number $i\geq 0$ of complete kinks, each of which transfers the SAP from a type-2 configuration to a type-1 configuration and back to a type-2 configuration.
	
If $i$ stands for the number of kinks, then $2i$ is the number of interior $x$-coordinates at which a kink of the SAP starts or ends. These $x$-coordinates can be chosen from the $l-1$ interior $x$-coordinates, where $l$ denotes the $x$-span of the overall SAP, i.e. the difference between the maximum and minimum $x$-coordinates of the SAP. Once the chosen coordinates are arranged increasingly,
	consecutive pairs necessarily form the start and end coordinates of
	the successive kinks. Hence, the start and end coordinates of the
	kinks can be chosen in
	$$
	\binom{l-1}{2i}
	$$
	ways.
	
	Furthermore, for each kink, the two parallel horizontal segments can
	lie either on rows $1$ and $2$ or on rows $0$ and $1$. This gives
	$2^i$ row choices for fixed start and end coordinates of the $i$ kinks.
	
	Conversely, every choice of $2i$ interior $x$-coordinates, together
	with a row choice for each of the $i$ kinks, determines exactly one
	SAP of this type.
	
	Thus, if $n$ stands for the length of the SAP under consideration,
	the number of such SAPs is
	$$
	\sum_{\substack{i\geq 0,\;\ l-1 \geq 2i\\
			2i+2l+4=n}}
	2^i\binom{l-1}{2i}.
	$$
	Indeed, over the span $l$ of the SAP, there are $2l$ horizontal steps,
	$2i$ additional vertical steps arising from the $i$ kinks, and four
	vertical steps arising from the left and right boundary configurations
	of the SAP.
	
	The equation $2i+2l+4=n$ gives
	$$
	l-1=\frac{n}{2}-i-3.
	$$
	Moreover, the necessary condition $l-1\geq 2i$ translates into
	$$
	\frac{n}{2}-i-3\geq 2i,
	$$
	or equivalently,
	$$
	i\leq\frac{n-6}{6}.
	$$
	Substituting for $l-1$ therefore yields
	$$
	p^{S_2}_{22}(n)
	=
	\sum_{i=0}^{\lfloor (n-6)/6\rfloor}
	2^i\binom{n/2-i-3}{2i},
	$$
	which proves the theorem.
\end{proof}

Furthermore, let us remark that the OEIS entry states that $a(m)$ also counts compositions of $m$ into parts of size $1$ of one kind and parts of size at least $3$ of two kinds, with no parts of size $2$ \cite[Sequence~A007909]{OEIS}. This interpretation becomes even clearer when we show that every SAP counted by $p^{S_2}_{22}(2m+6)$ can be encoded by a binary sequence in the following way.

The binary word associated with such a SAP is constructed as follows. Some steps at the left boundary and in the leftmost unit-width slice are already fixed and do not have to be encoded. In detail, since we consider a SAP with type-2 boundary configurations at both ends, it has two vertical steps at its leftmost $x$-coordinate, followed by two horizontal steps in the extreme rows $0$ and $2$.

From this point on, while traversing the SAP slice by slice from left to right, the bits are interpreted as follows. A $\mathtt{0}$ means that the horizontal steps from the preceding slice are simply extended in the same rows. A $\mathtt{1}$, on the other hand, has two possible meanings. If the two horizontal steps in the preceding slice lie in adjacent rows, then the steps in the next slice lie in the two extreme rows, so that the current kink ends. If the two horizontal steps in the preceding slice lie in the extreme rows, then the $\mathtt{1}$ marks the beginning of a kink, and the following bit $b$ decides whether the succeeding two horizontal steps lie in rows $1$ and $2$ or in rows $0$ and $1$: $b=1$ indicates the former and $b=0$ the latter case. This choice bit is an additional bit and does not correspond to another unit-width slice.

Hence, in the case of SAPs starting and ending with a type-2 configuration, the binary sequence has a unique decomposition into blocks of the form
\begin{itemize}
	\item $\mathtt{0}$;
	\item $\mathtt{1}b\mathtt{0}^{\,r}\mathtt{1}$, where
	$b\in\{0,1\}$ and $r\geq0$.
\end{itemize}
The words arising in this way are precisely the concatenations of these blocks, we refer to them as admissible binary words.This description makes it clear that every admissible binary word encodes exactly one self-avoiding polygon of the type under consideration, up to horizontal translation, and that every such self-avoiding polygon is encoded by exactly one admissible binary word. See Figure~\ref{fig:sap-24-span2-marked} for a self-avoiding polygon with type-2 boundary configurations at both ends and its corresponding binary encoding.

\begin{figure}[htbp]
	\centering
	\begin{tikzpicture}[x=0.7cm,y=0.7cm,line cap=round]
		\tikzset{
			gridline/.style={
				line width=0.8pt,
				draw=black!70
			},
			walk/.style={
				line width=2.4pt,
				draw=red!85
			},
			stepmark/.style={
				circle,
				draw=red!85,
				fill=white,
				line width=0.8pt,
				inner sep=0pt,
				minimum size=4.2pt
			}
		}
		
		\draw[gridline] (-0.5,0) -- (8.5,0);
		\draw[gridline] (-0.5,1) -- (8.5,1);
		\draw[gridline] (-0.5,2) -- (8.5,2);
		
		\draw[walk]
		(0,0) -- (1,0) -- (2,0) -- (3,0) -- (3,1)
		-- (4,1) -- (4,0) -- (5,0) -- (6,0) -- (7,0)
		-- (8,0) -- (8,1) -- (8,2) -- (7,2) -- (7,1)
		-- (6,1) -- (5,1) -- (5,2) -- (4,2) -- (3,2)
		-- (2,2) -- (1,2) -- (0,2) -- (0,1) -- cycle;
		
		\foreach \x/\y in {
			0/0,1/0,2/0,3/0,3/1,4/1,4/0,5/0,
			6/0,7/0,8/0,8/1,8/2,7/2,7/1,6/1,
			5/1,5/2,4/2,3/2,2/2,1/2,0/2,0/1
		}
		\node[stepmark] at (\x,\y) {};
	\end{tikzpicture}
	
	\caption{A self-avoiding polygon of length $24$ with type-2 boundary
		configurations at both extreme $x$-coordinates. Its binary code is
		$\mathtt{001111001}$, with the unique block decomposition
		$\mathtt{0|0|111|1001}$. The blocks $\mathtt{111}$ and $\mathtt{1001}$
		encode the upper and lower kink, respectively.}
	\label{fig:sap-24-span2-marked}
\end{figure}
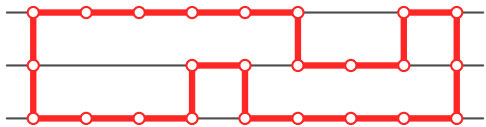

There is therefore a one-to-one correspondence between admissible binary words of length $m$ and self-avoiding polygons of length $2m+6$. Indeed, we start with an initial seed consisting of two vertical steps and two horizontal steps, while the two rightmost vertical steps are also not included in the binary encoding. Each bit associated with an interior $x$-coordinate accounts for two further horizontal steps, whereas each additional choice bit $b$ accounts for the two vertical steps of the corresponding kink. Thus, the four steps of the initial seed together with the two rightmost vertical steps contribute six steps, while the $m$ encoded bits contribute $2m$ further steps. For $m=0$, the empty binary word is admitted and encodes the unique SAP of length $6$.

Finally, the block $\mathtt{0}$ has length $1$ and is of one kind, whereas a block $\mathtt{1}b\mathtt{0}^{\,r}\mathtt{1}$ has length $r+3\geq3$ and is of two kinds, according to the value of $b$. Therefore, the number of compositions of $m$ into parts of size $1$ of one kind and parts of size at least $3$ of two kinds, with no parts of size $2$, equals $p^{S_2}_{22}(2m+6)$, the number of self-avoiding polygons of length $2m+6$ that start and end with a type-2 configuration.

\end{document}